\documentclass[12pt]{amsart}
\usepackage[]{hyperref}
\usepackage{amssymb,amsfonts,amsmath,amsopn,amstext,amscd,latexsym, amsthm, enumerate,mathrsfs}

\usepackage[margin=36mm]{geometry}
\usepackage{bbm}
\usepackage[all,cmtip]{xy}

\usepackage{enumerate}
\usepackage[shortlabels]{enumitem}

\newtheorem{theorem}{Theorem}[section]

\newtheorem{lemma}[theorem]{Lemma}

\newtheorem*{thmA}{Theorem~A}
\newtheorem*{thmB}{Theorem~B}

\newtheorem*{example}{Examples}

\newtheorem*{question}{Question}
\theoremstyle{remark}

\DeclareMathOperator{\Irr}{Irr}
\DeclareMathOperator{\IBr}{IBr}

\DeclareMathOperator{\SL}{SL}

\newcommand{\cC}{{\mathcal C}}
\newcommand{\cD}{{\mathcal D}}

\newcommand{\bbZ}{{\mathbb Z}}

\newcommand{\Ker}{{\mathrm {Ker}}}
\newcommand{\Syl}{{\mathrm {Syl}}}

\DeclareMathOperator{\Sym}{S}

\DeclareMathOperator{\PSL}{PSL}

\newcommand{\cd}{{\mathrm {cd}}}

\newcommand{\OO}{\mathbf{O}}
\newcommand{\Centralizer}{\mathbf{C}}

\newcommand{\N}{\mathbf{N}}

\numberwithin{equation}{section}

\begin{document}

\title[Brauer characters of $q^\prime$-degree]{Brauer characters of $q^\prime$- degree}

\author[M. L. Lewis]{Mark L. Lewis}
\address{Department of Mathematical Sciences, Kent State University, Kent, OH 44242, USA}
\email{lewis@math.kent.edu}

\author[H. P. Tong-Viet]{Hung P. Tong Viet}
\address{Department of Mathematical Sciences, Kent State University, Kent, OH 44242, USA}
\email{htongvie@math.kent.edu}

\thanks{}


\date{\today}

\date{\today}


\begin{abstract}
We show that if $p$ is a prime and $G$ is a finite $p$-solvable group satisfying the condition that a prime $q$ divides the degree of no irreducible $p$-Brauer character of $G$, then the normalizer of some Sylow $q$-subgroup of $G$ meets all the conjugacy classes of $p$-regular elements of $G$.
\end{abstract}

\maketitle

\section{Introduction}

Throughout this paper, $G$ will be a finite group and $p$ will be a prime. Let $\Irr (G)$ be the set of all complex  irreducible characters of $G$ and  let $\IBr (G)$ be the set of irreducible $p$-Brauer characters of $G$.  The celebrated It\^{o}-Michler theorem says that $p$ does not divide $\chi (1)$ for all $\chi \in \Irr (G)$ if and only if $G$ has a normal abelian Sylow $p$-subgroup.

Many variations of this theorem have been proposed and studied in the literature.  See the recent survey paper by G. Navarro on this topic in \cite{N1}.  One might ask whether there is any version of It\^{o}-Michler theorem for Brauer characters of finite groups. Now let $q$ be a prime and assume that $q$ divides the degree of no irreducible $p$-Brauer character of $G.$  Indeed, it is known that if $q=p,$ then $G$ has a normal Sylow $q$-subgroup.  (See Theorem 3.1 of \cite{N1}.)

This raises the question of whether there exists a similar result when $q\ne p$.  In particular, in Problem 3.2 of \cite{N1}, Navarro asks when $G$ is a $p$-solvable group and $q$ divides the degree of no irreducible $p$-Brauer character, is it true that every $p$-regular conjugacy class of $G$ intersects the normalizer of a Sylow $q$-subgroup of $G?$  In our first result, we prove that this is true.

\begin{thmA}\label{th:main1}
Let $p$ be a prime and let $G$ be a finite $p$-solvable group. Let $q$ be a prime and suppose that $q$ divides the degree of no irreducible $p$-Brauer character of $G.$ Then every $p$-regular conjugacy class of $G$ meets $\N_G (Q),$ where $Q$ is a Sylow $q$-subgroup of $G.$
\end{thmA}

The conclusion of Theorem A can be restated in the language of permutation group theory as follows. Let $H$ be a proper subgroup of a finite group $G.$ Following \cite{IKLM}, we say that an element $x\in G$ is an \emph{$H$-derangement} in $G$ if the conjugacy class $x^G$ containing $x$ does not meet $H.$ We write $\Delta_H(G)$ for the set of all $H$-derangements of $G.$ If $H$ is core-free in $G,$ then $G$ is a permutation group acting on the right coset space $\Omega=G/H$ with point stabilizer $H$ and $\Delta(G)=\Delta_H(G)$ is the set of all derangements or fixed-point-free elements of $G$ on $\Omega.$

A classical theorem due to Jordan \cite{Jordan} says that the set $\Delta_H(G)$ is non-empty. Notice that $$\Delta_H(G)=G\setminus \cup_{g\in G}H^g.$$  Derangements have many applications in topology and number theory (see \cite{Serre}). Derangements have been used in studying zeros of ordinary character theory. Now with this concept, Theorem A can be restated as follows.

\medskip
\emph{Let $p$ and $q$ be primes and let $G$ be a  finite $p$-solvable group and $Q$ a Sylow $q$-subgroup of $G.$ If $q$ divides the degree of no irreducible $p$-Brauer character of $G,$ then all $\N_G(Q)$-derangements of $G$ have order divisible by $p$.}

\medskip
As already mentioned in \cite{N1}, using the It\^{o}-Michler Theorem for Brauer characters and a result in \cite{FKS} which states that every finite permutation group of degree $>1$ contains a derangement of prime power order, it is easy to see that for a finite group $G,$ every $p$-Brauer character of $G$ has $p'$-degree if and only if every $\N_G(P)$-derangement of $G,$ for some Sylow $p$-subgroup $P$ of $G,$ has order divisible by $p.$ That is, Theorem A holds for all finite groups when $q=p.$ Unfortunately, this does not hold true when $q$ is different from $p$.  We will provide examples to show that the $p$-solvable assumption on $G$ in Theorem A is necessary.

Returning to Problem 3.2 of \cite{N1}, we note that Navarro asks for a characterization of $p$-solvable groups where $q \ne p$ and $q$ does not divide the degree of any irreducible $p$-Brauer character of $G$.  Our next goal is to obtain just such a characterization.  Manz and Wolf studied these groups in \cite{MW2}.  Many of the results of that paper can be found in Section $13$ of \cite{MW1}.  In Corollary $13.15$ of \cite{MW1}, they prove that ${\bf O}^{q'} (G)$ is solvable.  When $G$ has an abelian Sylow $q$-subgroup, it turns out that this additional condition is sufficient.  In particular, we prove the following.

\begin{thmB}\label{th:abelian1}
Let $p$ and $q$ be distinct primes and suppose that $G$ is $p$-solvable and a Sylow $q$-subgroup $Q$ of $G$ is abelian. Then $q\nmid\varphi(1)$ for all $\varphi\in\IBr(G)$ if and only if the following conditions hold:
\begin{enumerate}[$(1)$]
\item $x^G\cap \N_G(Q)\neq\emptyset$ for all $p$-regular elements $x\in G;$
\item $\OO^{q'}(G)$ is solvable;
\end{enumerate}
\end{thmB}

Now, if we could prove that $G$ has to have an abelian Sylow $q$-subgroup, then this would answer Navarro's question.  However, it is not difficult to find examples where the Sylow $q$-subgroups are not abelian.   We will present a characterization for $p$-solvable groups $G$ that have $q \nmid \varphi (1)$ for all $\varphi \in \IBr (G)$ for primes $p \ne q$, however it is so technical that we delay the statement of the result until later (see Theorem \ref{th:characterization}).

\section{Proof of Theorem A}

Throughout this section, $p$ and $q$ are distinct primes and $G$ is a finite group. Suppose that $q$ divides the degree of no irreducible $p$-Brauer character of $G$ and let $N$ be a normal subgroup of $G.$ Since $\IBr(G/N)\subseteq\IBr(G)$ and by \cite[Corollary~8.7]{Navarro}, we see that both $N$ and $G/N$ satisfy this property.  In particular,  all $p$-Brauer characters of $\OO^{q'}(G)$ have $q'$-degree.  We show the converse of this holds when $G/N$ is $p$-solvable.

\begin{lemma}\label{lem1} Let $p$ and $q$ be distinct primes and
suppose that $G/\OO^{q'}(G)$ is $p$-solvable. Then $q\nmid \varphi(1)$ for all $\varphi\in\IBr(G)$ if and only if $q\nmid \beta(1)$ for all $\beta\in\IBr(\OO^{q^\prime}(G)).$
\end{lemma}

\begin{proof} By the discussion above, it suffices to show that  if all irreducible $p$-Brauer characters of $L:=\OO^{q^\prime}(G)$ have $q'$-degree, then $q\nmid \varphi(1)$ for all $\varphi\in\IBr(G)$. Let $\varphi\in\IBr(G)$ and let $\theta\in\IBr(L)$ be an irreducible constituent of $\varphi_L.$ By \cite[Theorem~8.30]{Navarro}, we have ${\varphi(1)}/{\theta(1)}$ divides $|G/L|.$ Since $G/L$ is a $q'$-group and $q\nmid \theta(1)$ by our assumption, we deduce that $q\nmid \varphi(1).$
\end{proof}

In Theorem $13.8$ of \cite{MW1}, Manz and Wolf prove in a $q$-series for $G$ where $G$ is $p$-solvable and $q$ divides the degree of no irreducible $p$-Brauer character of $G$, the $q$-factors are abelian that the $q$-length of $G/{\bf O}_{p,q} (G)$ is at most $1$ and that the Sylow $q$-subgroups are metabelian.  We next formally state the results of Theorem 13.8 and Corollary 13.15 of \cite{MW2}.

\begin{lemma}\label{lem2} Let $p$ and $q$ be distinct primes and let $G$ be a finite $p$-solvable group. Suppose that $q\nmid \varphi(1)$ for all $\varphi\in\IBr(G).$ Then the following hold.

\begin{enumerate}
\item $\OO^{q'}(G)$ is solvable and so $G$ is $q$-solvable.
\item In each $q$-series of $G,$ the $q$-factors are abelian and the Sylow $q$-subgroup of $G$ is metabelian.
\item $G/\OO_{p,q}(G)$ has $q$-length at most $1.$
\end{enumerate}

\end{lemma}

Fix a prime $p$. Let $G$ be a finite group and let $H$ be a proper subgroup of $G.$ For brevity, we say that the pair $(G,H)$ has property $\mathcal{D}_p$ if $x^G\cap H$ is not empty for all $p$-regular elements $x\in G$ or equivalently all $H$-derangements of $G$ have order divisible by $p.$

The following slightly generalizes Lemma 4.2 in \cite{BT}.

\begin{lemma}\label{lem:normal}
Let $H$ be a  proper subgroup of a finite group $G$ and $L$ be a normal subgroup of $G$ such that $G=HL$.  If $T$ is a proper subgroup of $L$ containing $H\cap L$,
then $\Delta_T(L)\subseteq \Delta_H(G)$.
\end{lemma}

\begin{proof}
Let $x\in \Delta_T(L)$ and assume that $x\not\in\Delta_H(G)$. Then  $x^{g}\in H$ for some $g\in G.$ Since $x\in L\unlhd G,$ we  have $x^{g}\in L,$ so $x^{g}\in H\cap L\le T.$  As $g\in G=HL=LH$, we can write $g=lh$ with $h\in H$ and $l\in L.$ Then $x^g=x^{lh}=(x^l)^h\in H$ which implies that $x^l\in H$ and since both $x$ and $l$ are in $L,$ we obtain that $x^l\in H\cap L\le T,$ which is a contradiction.
\end{proof}

We collect in the next lemma some properties of finite groups satisfying $\mathcal{D}_p$.

\begin{lemma}\label{D-property} Let $p$ be a prime, $G$ be a finite group and $H$ be a subgroup of $G.$ Let $L\unlhd G.$
\begin{enumerate}[$(1)$]
\item  If $G=HL$ and $(G,H)$ satisfies $\mathcal{D}_p$ then so does $(L,H\cap L).$

\item If $L$ is a $p$-group or $p'$-group, $(G,H)$ satisfies $\mathcal{D}_p$ and $G\neq HL,$ then $(G/L,HL/L)$ also satisfies $\mathcal{D}_p.$

\item If $H\le K< G$ and $(G,H)$ satisfies $\mathcal{D}_p$, then $(G,K)$ satisfies $\mathcal{D}_p.$

\item If $L\unlhd G$ such that $L\leq H$ and $(G/L,H/L)$ satisfies $\mathcal{D}_p$ then $(G,H)$ satisfies $\mathcal{D}_p.$
\end{enumerate}
\end{lemma}

\begin{proof}
(1) This follows immediately since $\Delta_{H\cap L}(L)\subseteq \Delta_H(G)$ by Lemma \ref{lem:normal}

\medskip
(2) Let $\overline{G}=G/L.$ Since $G\neq HL,$ we see that $\overline{H}$ is a proper subgroup of $\overline{G}.$ Let $\overline{x}\in\Delta_{\overline{H}}(\overline{G}).$ Then $\overline{x}^{\overline{G}}\cap \overline{H}=\emptyset$ which implies that $x^G\cap HL=\emptyset$ and thus $x^G\cap H=\emptyset$ or $x\in\Delta_H(G)$ so $p$ divides $|x|,$ the order of $x.$

If $L$ is a $p'$-group, then the order of $\overline{x}$ must be divisible by $p$ and we are done. So, assume that $L$ is a $p$-group and that the order of $\overline{x}$, say $n,$ is indivisible by $p$. As $L$ is a $p$-group and $p\nmid n,$ we see that $|x|=p^an$ for some integer $a.$ There exist integers $u,v$ such that $1=up^a+vn.$ Hence $x=(x^{p^a})^u(x^n)^v$ where $x^n\in L.$ Thus $\overline{x}=\overline{y}$ with $y=x^{up^a}$ and $|y|=n.$ We now have that $\overline{y}^{\overline{G}}\cap \overline{H}=\emptyset$ and thus $y^G\cap HL=\emptyset$ so $y\in\Delta_H(G)$ and hence $p$ divides $|y|,$ which is a contradiction.

\medskip

(3) is obvious since $\Delta_K(G)\subseteq \Delta_H(G).$

\medskip
(4) Let $x\in\Delta_H(G).$ Then $x^G\cap H=\emptyset$ and thus $\overline{x}^{\overline{G}}\cap \overline{H}=\emptyset$ since $\overline{H}=H/L.$ So $p$ divides $|x|$ as  $(G/L,H/L)$ satisfies $\mathcal{D}_p$. Clearly $p$ must divide $|x|$ and we are done.
\end{proof}

We now apply Lemma \ref{D-property} for $H=\N_G(Q),$ where $Q$ is a Sylow $q$-subgroup of $G.$ The next lemma asserts that the condition `$\N_G(Q)$ meets every $p$-regular class of a finite group $G$'  is inherited to normal subgroups.

\begin{lemma}\label{lem6}
Let $p$ and $q$ be distinct primes. Let $Q$ be a Sylow $q$-subgroup of $G$ and let $L\unlhd G$. Suppose that $x^G\cap \N_G(Q)\neq\emptyset$ for all $p$-regular elements $x$ of $G.$ Then $x^L\cap \N_L(Q\cap L)\neq\emptyset$ for all $p$-regular elements $x$ of $L.$ In particular, if $Q\leq L,$ then $x^L\cap \N_L(Q)\neq\emptyset$ for all $p$-regular elements $x$ of $L.$
\end{lemma}

\begin{proof} Let $H=\N_G(Q)$ and $U=Q\cap L.$ Then $U\unlhd Q\unlhd H\le \N_G(U)$ and $U\in\Syl_q(L).$ Since $L\unlhd G,$ we have $G=\N_G(U)L$ by Frattini's argument.

If $U\unlhd L,$ then the conclusion is trivially true. So, we may assume that $\N_L(U)$ is a proper subgroup of $L$ which implies that both $H$ and $\N_G(U)$ are proper subgroups of $G.$ It suffices to show that the pair $(L,\N_L(U))$ satisfies $\mathcal{D}_p.$

Clearly, $(G,H)$ satisfies $\mathcal{D}_p$ by the hypothesis, so $(G,\N_G(U))$ satisfies $\mathcal{D}_p$ by Lemma \ref{D-property}$(3).$ Now part (1) of Lemma \ref{D-property} implies that $(L,L\cap \N_G(U))$ satisfies $\cD_p$ or $(L, \N_L(U))$ satisfies $\cD_p$ as wanted.
\end{proof}

We next prove Theorem A under the additional hypothesis that $G = Q\OO_{q'} (G)$ where $Q$ is a Sylow $q$-subgroup of $G$.

\begin{lemma}\label{lem3}
 Let $p$ and $q$ be distinct primes and let $Q\in\Syl_q(G)$. Suppose that $G=Q\OO_{q'}(G)$ and that $q\nmid \varphi(1)$ for all $\varphi\in\IBr(G).$ Let $K=\OO_{q'}(G)$ and $H=\N_G(Q).$ Then
 \begin{enumerate}
 \item $Q$ is abelian and $x^K\cap\Centralizer_K(Q)\neq\emptyset$ for all $p$-regular elements $x\in K;$
 \item $x^G\cap \N_G(Q)$ is non-empty for all $p$-regular elements $x\in G$.
 \end{enumerate}
\end{lemma}

\begin{proof}
Assume first that $q\nmid \varphi(1)$ for all $\varphi\in\IBr(G).$  Since $G=QK,$ with $K\unlhd G$ and $Q\cap K=1,$ we deduce that $H=Q\Centralizer_K(Q).$

\medskip
$(0)$ $Q$ is abelian. This follows from Lemma \ref{lem2}$(ii).$

\medskip

$(1)$ Every $\theta\in\IBr(K)$ is $Q$-invariant. Let $\theta\in\IBr(K)$ and let $T=I_G(\theta).$ By Clifford correspondence \cite[Theorem~8.9]{Navarro}, if $\psi\in\IBr(T\vert\theta),$ then $\psi^G\in\IBr(G)$ and so $\psi^G(1)=|G:T|\psi(1).$ As $K\unlhd T\le G$ and $q\nmid \psi^G(1),$ we must have that $T=G$ or equivalently $\theta$ is $G$-invariant and hence it is $Q$-invariant.

\medskip
$(2)$ $Q$ stabilizes all $p$-regular conjugacy classes of $K.$ From $(1),$ $Q$ stabilizes all irreducible $p$-Brauer characters of $K$ and hence it stabilizes all the projective indecomposable characters $\Phi_\mu$ associated with $\mu\in\IBr(K).$ From \cite[Theorem~2.13]{Navarro}, the set \[\{\Phi_\mu \mid \mu \in \IBr(K)\}\] is a basis of the space of complex functions of $K$ vanishing off the set $K^\circ$ of all $p$-regular elements of $K.$ Therefore $Q$ stabilizes the characteristic functions of the $p$-regular conjugacy classes of $K$ and thus $Q$ stabilizes all the $p$-regular conjugacy classes of $K.$

\medskip
$(3)$ If $x\in K$ is $p$-regular, then $x^K\cap \Centralizer_K(Q)\neq\emptyset.$ Let $\cC$ be a $p$-regular class of $K.$  By $(2),$ $Q$ acts  on $\cC$, $K$ acts transitively on $\cC$ and $Q$ acts coprimely on $K.$ By  Corollary ~1 of Theorem~4 in \cite{Glauberman}, we obtain that $\cC\cap \Centralizer_K(Q)\neq\emptyset.$ This proves $(i).$

\medskip
$(4)$ $y^G\cap H\neq\emptyset$ for every $p$-regular element $y\in G.$ Let $y$ be a $p$-regular element of $G$. If $y\in K,$ then $y^K\cap \Centralizer_K(Q)\neq\emptyset$ so $y^G\cap H\neq\emptyset$ and we are done. So, we can assume that $y\not\in K.$

Suppose that the order of $y$ is $q^a\cdot m$, where $a,m\ge 1$ and $q\nmid m.$ Since $\gcd(q^a,m)=1$, there exist integers $u,v$ such that $1=uq^a+vm.$ Let $y_1=y^{vm}$ and $y_2=y^{uq^a}.$ Then $y=y_1y_2=y_2y_1,$ where  $y_1$ is a $q$-element and $y_2$ is $q'$-element. By Sylow theorem, $y_1^t\in Q$ for some $t\in G.$ Since $y_2^t$ is a $q'$-element and $K\unlhd G$ is a Hall $q'$-subgroup of $G,$ we deduce that $y_2^t\in K.$ Thus $y^t=y_1^ty_2^t=y_2^ty_1^t,$ where $y_1^t\in Q$ and $y_2^t\in K.$ Replacing $y$ by $y^t,$ we can assume $y=y_1y_2=y_2y_1$ with $y_1\in Q$, $y_2\in K$. Since $y$ is $p$-regular, we see that $y_2\in K$ is also $p$-regular and thus $y_2^k\in\Centralizer_K(Q)$ for some $k\in K$ by (3) above. Hence $Q^{k^{-1}}\leq \Centralizer_G(y_2).$ By Sylow theorem, there exists $l\in \Centralizer_G(y_2)$ such that $Q^{k^{-1}}=Q^l$ so that $Q^{lk}=Q$ or equivalently $lk\in H.$

Since $y=y_1y_2$ and $y_1\in Q\leq H,$ we deduce that $y_1^{lk}\in H.$ Moreover, as $l\in\Centralizer_G(y_2),$ we have $y_2^{lk}=y_2^k\in \Centralizer_K(Q)\leq H.$ Therefore,  $$y^{lk}=y_1^{lk}y_2^{lk}\in H.$$
So, we have shown that $y^G\cap H\neq\emptyset$ for all $p$-regular elements $y\in G.$ This completes the proof of $(ii).$
\end{proof}

We are now ready to prove Theorem A which we restate here.

\begin{theorem}
Let $p$ and $q$ be distinct primes and let $G$ be a finite $p$-solvable group. Suppose that $q\nmid \varphi(1)$ for all $\varphi\in\IBr(G).$ Then $x^G\cap \N_G(Q)\neq\emptyset$ for all $p$-regular elements $x\in G,$ where $Q$ is a Sylow $q$-subgroup of $G.$
\end{theorem}

\begin{proof} Let $Q\in\Syl_q(G)$ and $H=\N_G(Q).$ Suppose that $q\nmid\varphi(1)$ for all $\varphi\in\IBr(G).$
We proceed by induction on $|G|.$

\medskip
(1) If $N\unlhd G$ is nontrivial, then $y^G\cap HN\neq\emptyset$ for all $p$-regular elements $y\in G.$  Since $G/N$ satisfies the hypothesis of the theorem with $|G/N|<|G|,$ by induction we deduce that $\N_{G/N}(QN/N)=\N_G(Q)N/N=HN/N$ meets all the $p$-regular classes of $G/N.$ It follows that  $y^G\cap HN\neq\emptyset$ for every $p$-regular element $y\in G.$

\medskip
(2) $\OO_p(G)=1=\OO_q(G)$ and $H_G=1.$  If $H_G\unlhd G$ is nontrivial, then the conclusion of the theorem holds by applying (1). So, we can assume $H_G=1$ and so $\OO_q(G)=1.$

Suppose that $\OO_p(G)$ is nontrivial. By (1) again, we see that $y^G\cap H\OO_p(G)$ is nonempty for every $p$-regular element $y\in G.$ Replacing $y$ by its conjugate if necessary, we can assume $y\in H\OO_p(G).$ Since $y$ is an element of $p'$-order and $H$ contains a Hall $p'$-subgroup $T$ of $H\OO_p(G),$ some $H\OO_p(G)$-conjugate of $y$ lies in $T\subseteq H$ and thus $y^G\cap H$ is nonempty.

\medskip

(3) Let $L=\OO^{q'}(G).$ Then $L=QK$ where $K=\OO_{q'}(L)$ is solvable. Since $L\unlhd G$ and $\OO_p(G)=1=\OO_q(G)$ by (2), we deduce that $\OO_p(L)=\OO_q(L)=1$ so that $\OO_{p,q}(L)=1.$ Therefore, by Lemma \ref{lem2} $L$ is solvable and has $q$-length at most $1.$ Since $L=\OO^{q'}(L),$ we must have that $L=Q\OO_{q'}(L)$ as wanted.

\medskip
(4) $x^K\cap \Centralizer_K(Q)\neq\emptyset$ for every $p$-regular element $x\in K.$ This follows from (3) and Lemma \ref{lem3}(i).

\medskip

(5) $y^G\cap H\neq\emptyset$ for every $p$-regular element $y\in G.$ Since $L=QK$ is solvable  by (3) and $\OO_q(L)=1$ by (2), we see that $K$ is nontrivial and solvable. Notice that $K$ is a solvable normal subgroup of $G.$ Hence $K$ contains a minimal normal subgroup of $G,$ say $N.$ Then $N$ is an elementary abelian $r$-group for some prime $r$ different from both $p$ and $q.$

From (1), we have that $y^G\cap HN\neq\emptyset$ for all $p$-regular elements $y\in G.$ Hence it suffices to show that $y^G\cap H\neq\emptyset$ for all $p$-regular elements $y\in HN.$ Now fix a $p$-regular element $y\in HN.$ Then $y=hn$ for some $h\in H$ and $n\in N.$ If $n=1,$ then $y=h\in H$ and we are done. So we assume that $n$ is nontrivial. By (4), $n^k\in\Centralizer_K(Q)$ for some $k\in K.$ Hence $Q^{k^{-1}}\le \Centralizer_{QK}(n)$, and thus by Sylow theorem, $Q^{k^{-1}}=Q^l$ for some $l\in \Centralizer_{QK}(n)$ so $Q^{lk}=Q$ and hence $lk\in H.$ Since $n^{lk}=n^k$ and $h^{lk}\in H,$ we obtain that  \[y^{lk}=(hn)^{lk}=h^{lk}n^{lk}=h^{lk}n^k\in H.\] Therefore, we have shown that $y^G\cap H$ is not empty.
\end{proof}

\section{Proof of Theorem B}

Let $N\unlhd G$ and let $\theta\in\IBr(N).$ Denote by $\IBr(G \mid \theta)$ the set of all irreducible $p$-Brauer characters of $G$ lying over $\theta.$ We now consider the converse of Lemma \ref{lem3}.

\begin{lemma}\label{lem4} Let $p$ and $q$ be distinct primes and suppose that $G=QK$ where $Q$ is an abelian Sylow $q$-subgroup of $G$ and $K=\OO_{q'}(G).$  If $Q$ is abelian and $x^K\cap\Centralizer_{K}(Q)\neq\emptyset$ for all $p$-regular elements $x\in K,$ then $q\nmid\beta(1)$ for all $\beta\in\IBr(G).$
\end{lemma}

\begin{proof} As $Q\cong G/K$ is abelian, all $\varphi\in\IBr(G/K)$ are linear. Since $(|K|,|G:K|)=1$, by \cite[Theorem~8.13]{Navarro} it suffices to show that every $\beta\in\IBr(K)$ is $Q$-invariant and thus is $G$-invariant;  hence $\beta$ is extendible to $G$ and thus $q$ divides the degree of no irreducible $p$-Brauer character of $G.$

By Corollary 1 of Theorem 4 in \cite{Glauberman}, $Q$ fixes all the $p$-regular conjugacy classes of $K$ since $Q$ fixes some element in every $p$-regular class of $K.$ Reversing the argument in the proof of Lemma \ref{lem3}, we see that $Q$ fixes all the characteristic functions of the $p$-regular classes of $K$ and thus $Q$ fixes all the indecomposable projective characters $\Phi_\mu$ where $\mu\in\IBr(K).$ Hence $Q$ fixes all the $p$-Brauer characters of $K.$
\end{proof}

We now present Theorem B which we restate here.

\begin{theorem}\label{th:abelian}
Let $p$ and $q$ be distinct primes and suppose that $G$ is $p$-solvable and a Sylow $q$-subgroup $Q$ of $G$ is abelian. Then $q\nmid\varphi(1)$ for all $\varphi\in\IBr(G)$ if and only if the following conditions hold:
\begin{enumerate}[$(1)$]
\item $x^G\cap \N_G(Q)\neq\emptyset$ for all $p$-regular elements $x\in G;$
\item $\OO^{q'}(G)$ is solvable;
\end{enumerate}
\end{theorem}

\begin{proof}
The `only if' direction of the theorem holds by combining Lemma \ref{lem2} and Theorem A. We now focus on the `if' part of the theorem.

Suppose that $Q\in\Syl_q(G)$ is abelian and conditions (1)-(2) hold.  Notice that $Q\leq \OO^{q'}(G)$.  In view of Lemma \ref{lem6}, condition (1) hold for $\OO^{q'}(G).$  Applying Lemma \ref{lem1}, we can assume that $G=\OO^{q'}(G),$ which implies that $G$ is solvable.  As $\OO_p(G)$ is contained in the kernel of all irreducible $p$-Brauer characters of $G,$ we can assume that $\OO_p(G)=1.$

As $G$ is solvable, we have from the Hall-Higman theorem that $$\Centralizer_{G/\OO_{q'}(G)}(\OO_{q',q}(G)/\OO_{q'}(G))\le \OO_{q',q}(G)/\OO_{q'}(G).$$ Since $Q$ is abelian,  $[Q,\OO_{q',q}(G)]\leq\OO_{q'}(G)$ and so $Q\le \OO_{q',q}(G)$ which implies that $G=Q\OO_{q'}(G)$ as $\OO^{q'}(G)=G.$ Now the result follows by Lemma \ref{lem4}.
\end{proof}

We will provide examples of groups that meet the conclusion of Theorem A and the conditions of Manz and Wolf, yet have irreducible $p$-Brauer characters whose degrees are divisible by $q$.  Thus, to state our characterization, we need one further condition beyond the one stated in Theorem A and the conditions found by Manz and Wolf.  

To state this condition, we introduce the following notation. Let $M\unlhd G$ and $N\unlhd M.$  We define $$\Centralizer_G(M/N)=\{g\in G\,:\, [g,M]\subseteq N\}.$$  Note that we are not assuming that $N$ is normal in $G$.  If $M/N$ is abelian, then $M$ is contained in $\Centralizer_G(M/N)$.  


\begin{lemma}\label{lem5} Let $M\unlhd G$ and let $N\unlhd M.$

\begin{enumerate}[$(1)$]
\item $\Centralizer_G(M/N)$ is a subgroup of $G.$
\item $N \unlhd \Centralizer_G (M/N).$
\item If $M\le K\le G$ and $K'\leq N,$ then $K\le \Centralizer_G(M/N).$
\end{enumerate}

\end{lemma}

\begin{proof}
For (1), it suffices to show that $\Centralizer_G(M/N)$ is closed under multiplication. Observe that if $g_1,g_2\in \Centralizer_G(M/N)$ and $x\in M,$ then \[[g_1g_2,x]=[g_2,x^{g_1}][g_1,x]\in [g_2,M]\cdot [g_1,M]\subseteq N\cdot N=N.\] So $g_1g_2\in \Centralizer_G(M/N).$  Since $N$ is normal in $M$, we have $[N,M] \le N$, so $N \le \Centralizer_G (M/N)$.   Notice that $[\Centralizer_G(M/N),N] \le [\Centralizer_G(M/N),M] \le N$, so $N$ is normal in $\Centralizer_G(M/N)$ yielding (2).  For (3), suppose that  $M\leq K\le G$ with $K'\le N.$ Then $K\leq \Centralizer_G(M/N)$ since $[K,M]\le [K,K]=K'\le N.$
\end{proof}


We now turn to our characterization of the $p$-solvable groups $G$ where $q \nmid \varphi (1)$ for all $\varphi \in \IBr (G)$ and $q \ne p$.  When $G$ is a finite group and $p$ is a prime, $\OO_p(G)$ is contained in the kernel of all $p$-Brauer characters of $G.$  So there is no loss in assuming that $\OO_p(G)=1.$  We note that (1) is the condition from Theorem A, and condition (2) is the condition of Manz and Wolf that appeared in Theorem B.  Observe that condition (4) (a) implies that $Q' \le \OO_q (L)$. Combined with (3), this implies that $G$ has a metabelian Sylow $q$-subgroup and that a $q$-series for $G$ will have abelian $q$-factors.  It not difficult to see since $G/\OO_q (L)$ has an abelian Sylow $q$-subgroup that it has $q$-length at most $1$.  Hence, the conditions here imply the conditions of Manz and Wolf.   

\begin{theorem}\label{th:characterization}
Let $p$ and $q$ be distinct primes and suppose that $G$ is $p$-solvable with $\OO_p(G)=1$. Let $L=\OO^{q'}(G)$ and let $Q$ be a Sylow $q$-subgroup of $G$. Then $q\nmid\varphi(1)$ for all $\varphi\in\IBr(G)$ if and only if the following conditions hold:
\begin{enumerate}[$(1)$]
\item $x^G\cap \N_G(Q)\neq\emptyset$ for all $p$-regular elements $x\in G;$
\item $L$ is solvable;
\item $\OO_q (L)$ is abelian;
\item For every normal subgroup $N$ of $\OO_q(L)$ with $\OO_q(L)/N$ cyclic, the following hold:
  \begin{enumerate}
    \item there exists an element $g\in L$ such that $(Q^g)'\leq N$
    \item Every $p$-regular conjugacy class of $C/\OO_q (L)$ meets $\N_{C/\OO_q (L)}(Q^g/\OO_q (L))$, where $C = \Centralizer_L (\OO_q (L)/N).$
  \end{enumerate}
\end{enumerate}
\end{theorem}

\begin{proof}

(\textbf{A}) Suppose that $q\nmid \varphi(1)$ for all $\varphi\in\IBr(G).$ We will show that $(1)-(4)$ hold. From Lemma \ref{lem1}, $q\nmid \varphi(1)$ for all $\varphi\in\IBr(L)$ since $L\unlhd G.$ The first three conditions follow from Lemma \ref{lem2} and Theorem A.

We now prove (4). Since $\OO_p(G)=1,$ we have $\OO_p(L)=1$ and so $\OO_{p,q}(L)=\OO_q(L)$.  Thus by Lemma \ref{lem2} (iii), $L/\OO_q(L)$ has a normal $q$-complement $K/\OO_q(L)$ where $K=\OO_{q,q'}(L).$

Let $N\unlhd \OO_q(L)$ with $\OO_q(L)/N$ cyclic. As $\OO_q(L)$ is an abelian $q$-group, we have $$\IBr(\OO_q(L))=\Irr(\OO_q((L)))\cong \OO_q(L).$$ So, there exists a Brauer character $\theta\in\IBr(\OO_q(L))$ such that $N=\Ker(\theta).$ Let $J=I_L(\theta).$ For every Brauer character $\beta\in\IBr(J \mid \theta)$, we have that $\beta^L\in\IBr(L)$ and $q\nmid \beta^L(1)=|L:J|\beta(1).$ So $J$ must contain a Sylow $q$-subgroup of $L$, hence $Q^g\leq J$ for some $g\in L$ and $q\nmid \beta(1).$ Now $J/\OO_q(L)$ has a normal $q$-complement $J_1/\OO_q(L)$ and $J=Q^gJ_1$ so $Q^g\cap J_1=\OO_q(L)$ and $J/J_1\cong Q^g/\OO_q(L).$ We see that $\theta$ extends to $\lambda\in\IBr(J_1)$ and $\lambda$ is also extendible to $\theta_0\in\IBr(J)$ as $q\nmid \beta(1)$ for every $\beta\in\IBr(J \mid \theta)$. Since $\theta_0$ is linear and $Q^g/\OO_q(G)$ is abelian, we deduce that $$(Q^g)'\leq \Ker(\theta_0)\cap \OO_q(L)\le \Ker(\theta)=N.$$

\medskip
We next claim that $J=\Centralizer_L(\OO_q(L)/N).$ Let $C:=\Centralizer_L(\OO_q(L)/N).$ Suppose first that $g\in J.$ Then $\theta^g=\theta$ so $\theta^g(x)=\theta(x)$ for all $x\in \OO_q(L).$ It follows that $gxg^{-1}x^{-1}\in \Ker(\theta)=N$ for all $x\in\OO_q(L)$ or equivalently $[g^{-1},M]\subseteq N$ hence $g^{-1}\in C$ and thus $g\in C.$ Conversely, suppose that $g\in C.$ Then $g^{-1}\in C$ and thus $[g^{-1},\OO_q(L)]\subseteq \Ker(\theta).$ Then $gxg^{-1}x^{-1}\in \Ker(\theta)$ for all $x\in \OO_q(L)$ which implies that $\theta^g=\theta$ or $g\in J.$ Therefore $J=\Centralizer_L(\OO_q(L)/N)$ as wanted.

\medskip
From the previous discussion, we know that $\theta$ extends to $\theta_0\in\Irr(J).$ So $\theta_0\mu\in\IBr(J \mid \theta)$ for all $\mu\in\IBr(J/\OO_q(L))$ by \cite[Corollary 8.20]{Navarro}. As $(\theta_0\mu)^L\in\IBr(L)$ and $q\nmid (\theta_0\mu)^L(1)=|L:J|\mu(1)$, we deduce that $q\nmid \mu(1)$ for all $\mu \in \IBr(J/\OO_q(L)).$ By Theorem A, we see that the pair $(J/\OO_q(L),Q^g/\OO_q(L))$ satisfies the last part of condition (4).  Notice since $Q^g\le J$ that $Q^g/\OO_q (L)$ will be a Sylow $q$-subgroup of $J/\OO_q (L)$.  This completes the proof of condition (4).

\medskip
(\textbf{B})
For the converse, suppose that (1)-(4) hold. We will show that $q\nmid\varphi(1)$ for all Brauer characters $\varphi \in \IBr(G).$ By Lemma \ref{lem1}, we can assume that $G=L=\OO^{q'}(G).$ From (2), $G$ is solvable.  By (3), $\OO_q(G)$ is abelian.

Notice that (4) implies that $Q' \le \OO_q (G)$, and so, $G/\OO_q (G)$ has an abelian Sylow $q$-subgoup.  Also, $G/\OO_q(G)$ is solvable.  Furthermore, (1) holds for $G/\OO_q(G)$  by Lemma \ref{D-property}(2). Thus $q\nmid \beta(1)$ for all $\beta\in\IBr(G/\OO_q(G))$ by Theorem \ref{th:abelian}.

Now it suffices to show that every $\beta\in\IBr(G)$ which does not contain $\OO_q(G)$ in its kernel has $q'$-degree. Let $\beta\in\IBr(G)$ be such a character and let $\theta\in\IBr(\OO_q(G))$ be an irreducible constituent of $\beta$ when restricted to ${\OO_q (G)}$. Let  $N=\Ker(\theta).$ Observe that $\theta$ is a nontrivial character of $\OO_q(G).$  Since $\OO_q(G)$ is abelian, the quotient $\OO_q(G)/N$ is cyclic.  Let $C = \Centralizer_G(\OO_q(G)/N)$.  By Lemma \ref{lem5}, we know that $N$ is normal in $C$, and from the definition of $C$, we see that $M/N$ is central in $C/N$.  This implies that $C$ stabilizes $\theta$. Conversely, suppose that $x \in G$ stabilizes $\theta$.  Then $\theta ([m,x]) = \theta (m^{-1}) \theta (m^x) = \theta (m^{-1}) \theta (m) =  \theta (1)$ for all $m \in M$.  This implies that $[M,x] \le \Ker (\theta) = N$, and so, $x \in C$.  It follows that $C$ is the stabilizer of $\theta$ in $G$.

By (4) we have $(Q^g)'\leq N$ for some $g\in G$.  Applying Lemma \ref{lem5} (3), $C$ contains a Sylow $q$-subgroup $Q^g$ of $G$.  Since $Q^g/N$ is abelian, $\theta$ extends to a Sylow $q$-subgroup of the quotient $C/\OO_q (G).$  Furthermore, for every prime $r\neq q$ and $R/\OO_q (G)$ a Sylow $r$-subgroup of $C/\OO_q (G),$ we see that $\theta$ extends to $R$ since $(|\OO_q(G)|,|R:\OO_q(G)|)=1$ (see \cite[Theorem 8.23]{Navarro}).  By \cite[Theorem 8.29]{Navarro}, $\theta$ extends to $\theta_0\in\IBr(C)$ and thus all irreducible constituents of $\theta^C$ have the form $\theta_0\lambda$ where $\lambda\in\IBr(C/\OO_q(G))$ by \cite[Corollary 8.20]{Navarro}. Applying \cite[Theorem 8.9]{Navarro}, we conclude that $\beta = (\theta_0\eta)^G$ for some $\eta \in \IBr(C/\OO_q(G))$.

Now, in light of (4), the group $C/\OO_q(G)$ satisfies all the hypotheses of Theorem \ref{th:abelian}, so $q\nmid \lambda(1)$ for all $\lambda\in\IBr(C/\OO_q(G))$ and thus $q\nmid (\theta_0\eta)^G (1) = \beta (1)$.

Therefore, we can conclude that $q\nmid \varphi(1)$ for all $\varphi\in\IBr(G).$
\end{proof}

This gives a group theoretical characterization of finite $p$-solvable groups whose all irreducible $p$-Brauer characters have $q'$-degree. This answers positively Problem 3.2 in \cite{N1}.

\section{Examples}

In the example below, we show that the solvable assumption on $G$ in Theorem A is necessary.

Let $p$ be a prime. We refer the readers to \cite{Lubeck} for some basic information on modular representation theory of $\SL_2(p^f)$ in defining characteristic $p,$ where $f\ge 1.$

It follows from Remark 4.5 \cite{Lubeck} that every $p$-modular irreducible representation of $\SL_2(p)$ has degree $k+1$ for some $k$ with $0\le k\le p-1$ and for odd $p,$ it is unfaithful if and only if $k+1$ is odd. The $p$-modular irreducible representations of $\SL_2(p^f)$ are then obtained by using Steinberg's tensor product theorem (see \cite[Theorem~2.2]{Lubeck}). It follows that every $p$-Brauer character degree of $\SL_2(p^f)$ is a product of at most $f$ Brauer character degrees of $\SL_2(p).$

Now, if $p=2,$ then all irreducible $p$-Brauer characters of $\SL_2(2^f)$ have $2$-power degrees since every irreducible $p$-Brauer characters of $\SL_2(2)$ has degree $1$ or $2.$

Denote by $\cd_p(G)$ the degrees of irreducible $p$-Brauer characters of $G.$
Assume $p>2.$ Then $\cd_p(\SL_2(p))$ consists of all integers from $1$ to $p;$ and $\cd_p(\PSL_2(p))$ consists of all odd integers from $1$ to $p.$ Finally, $\cd_p(\SL_2(p^2))$ consists of all integers of the form $ab$ where $a,b\in\{1,2,\cdots,p\}.$

With these results on $p$-Brauer character degrees of $\SL_2(p^f)$ and $\PSL_2(p)$, we have the following.

\begin{example}
Let $p$ and $q$ be distinct primes and let $f\ge 1$ be an integer.
\begin{enumerate}
\item Assume $f\ge 4$, $p=2$, and $q$ is a prime divisor of $2^{f}+1.$ Let $G=\SL_2(2^f)$ and $Q\in\Syl_q(G).$ Then $\N_G(Q)\cong {\rm{D}}_{2(2^f+1)} $ and all irreducible $2$-Brauer characters of $G$ have $2$-power degree.  In particular, $q\nmid \varphi(1)$ for all $2$-Brauer characters $\varphi\in\IBr(G)$, but $\N_G(Q)$ contains no element of order $2^f-1.$

\item Let $p\ge 5$ be a prime and let $q$ be a prime divisor of $p^2+1$ such that $q>p.$ Let $G=\SL_2(p^2).$ Then $q$ divides the degree of no irreducible $p$-Brauer character of $G.$ However, $\N_G(Q)\cong {\rm{D}}_{p^2+1}\cdot 2$ contains no $p$-regular element of order $p^2-1.$

\item Let $p$ be a prime of the form $2^f\pm 1\ge 17$ and let $G=\PSL_2(p).$ Then all irreducible $p$-Brauer characters of $G$ have odd degree and that the Sylow $2$-subgroup $Q$ of $G$ is maximal in $G.$ Then $2\nmid \varphi(1)$ for all $\varphi\in\IBr(G)$ but $\N_G(Q)=Q$ contains no odd $p$-regular element of $G.$
\end{enumerate}
\end{example}

Next, we have examples of $p$-solvable groups $G$ where $q \nmid \varphi (1)$ for all $\varphi \in \IBr (G)$ and $G$ has a nonabelian Sylow $q$-subgroup.  The easiest example is to take $p = 3$, $q = 2$, and $G = \Sym_4$.  Let $V$ be the additive group of the field $\textbf{F}$ of order $3^3$, let $C$ be the subgroup of the multiplicative group of $\textbf{F}$ having order $13$, and let $A$ be the Galois group of $\textbf{F}$ over $\mathbb{Z}_3$.  Now, $C$ acts on $V$ and $A$ acts on $VC$, and we take $G$ to be $VCA$.  With $p = 13$ and $q = 3$, we see that the irreducible $13$-Brauer characters have degrees $1$ and $13$ and $G$ has a nonabelian Sylow $3$-subgroup.

Finally, we present examples of groups that satisfy the conclusion of Theorem A and the conditions of Manz and Wolf, yet have irreducible $p$-Brauer characters whose degrees are divisible by $q$.  We begin by noting that all $\{ p, q \}$-groups trivially satisfy the conclusion of Theorem A since the $p$-regular elements will have $q$-power order and hence necessarily be conjugate to elements of the given Sylow $q$-subgroup.  Also, by Burnside's theorem, we know that any $\{ p, q \}$-group is necessarily solvable.  Thus, it suffices to find a $\{ p, q \}$-group $G$ where in the $q$-series for $G$, the $q$-factors are abelian, the $q$-length of $G/{\bf O}_{p,q} (G)$ is at most $1$, and the Sylow $q$-subgroups are metabelian, and there exists a $p$-Brauer character whose degree is divisible by $q$. A specific example when $p = 3$ and $q = 2$ can be found by taking the semidirect product of $\Sym_3$ acting on two copies of the Klein $4$-group where the action found in $\Sym_4$.  Obviously, a Sylow $2$-subgroup is metabelian, the $2$-factors in $2$-series for $G$ will be abelian, and $G$ will have $2$-length $1$.  Finally, it is not difficult to see that there exist irreducible $3$-Brauer characters for $G$ that have degree $6$.  We note that there is nothing particular about $3$ and $2$ that are needed for an example.  We claim that for any two distinct primes $p$ and $q$, the iterated wreath product of $\bbZ_q$ by $\bbZ_p$ and then $\bbZ_q$ again will yield an example, but we leave the details to the reader.

\end{document}